\documentclass{amsart}
\usepackage{hyperref}
\usepackage{amsmath,amsfonts,amssymb,amsthm,bbm}

\usepackage{color}
\newcommand{\red}{\color{red}}

\newcommand{\R}{{\mathbb R}}

\newcommand{\N}{{\mathbb N}}

\makeatletter
\def\moverlay{\mathpalette\mov@rlay}
\def\mov@rlay#1#2{\leavevmode\vtop{%
   \baselineskip\z@skip \lineskiplimit-\maxdimen
   \ialign{\hfil$\m@th#1##$\hfil\cr#2\crcr}}}
\newcommand{\charfusion}[3][\mathord]{
    #1{\ifx#1\mathop\vphantom{#2}\fi
        \mathpalette\mov@rlay{#2\cr#3}
      }
    \ifx#1\mathop\expandafter\displaylimits\fi}
\makeatother

\newcommand{\bigcupdot}{\charfusion[\mathop]{\bigcup}{\cdot}}

\newcommand{\FF}{\mathcal{F}}
\DeclareMathOperator{\Dim}{dim}
\DeclareMathOperator{\sgn}{sgn}
\DeclareMathOperator{\inter}{int}
\DeclareMathOperator{\relint}{relint}
\DeclareMathOperator{\aff}{aff}
\DeclareMathOperator{\BL}{L}
\DeclareMathOperator{\pos}{pos}
\DeclareMathOperator{\ext}{ext}

\newcommand{\I}{\mathbbm{1}}

\theoremstyle{plain}
\newtheorem{theorem}{Theorem}[section]
\newtheorem{lemma}[theorem]{Lemma}

\newtheorem{claim}[theorem]{Claim}

\theoremstyle{definition}

\theoremstyle{remark}

\begin{document}

\author{Daniel Hug}
\address{Daniel Hug, Karlsruhe Institute of Technology (KIT),
Department of Mathematics, Englerstr.\ 2,
D-76128 Karlsruhe, Germany}
\email{daniel.hug@kit.edu}

\author{Zakhar Kabluchko}
\address{Zakhar Kabluchko, Institut f\"ur Mathematische Statistik,
Universit\"at M\"unster,
Or\-l\'e\-ans--Ring 10,
48149 M\"unster, Germany}
\email{zakhar.kabluchko@uni-muenster.de}

\title[Inclusion-exclusion identity for normal cones]{An inclusion-exclusion identity for normal cones of polyhedral sets}

\keywords{Inclusion-exclusion identity, normal cone, Euler relation, convex set, polytope, polyhedral set, convex cone}

\subjclass[2010]{Primary: 52A20, 52B11; secondary: 52A55}

\date{December 23, 2016}

\begin{abstract}
For a nonempty polyhedral set $P\subset \R^d$, let $\FF(P)$ denote the set of faces of $P$,
and let $N(P,F)$ be the normal cone of $P$ at the nonempty face $F\in\FF(P)$. We prove the identity
$$
\sum_{F\in\FF(P)}(-1)^{\Dim F}\I_{F-N(P,F)} =
\begin{cases}
1, &\text{if $P$ is bounded},\\
0, &\text{if $P$ is unbounded and line-free}.
\end{cases}
$$
Previously, this formula was known to hold everywhere outside some exceptional set of Lebesgue measure $0$
or for polyhedral cones.
The case of a not necessarily line-free polyhedral set is also covered by our general theorem.
\end{abstract}
\maketitle

\section{Introduction and statement of the result}
\subsection{Polyhedral sets}
A \emph{polyhedral set} $P$ in  $\R^d$  is an intersection of finitely many closed half-spaces. That is, $P$ is a closed convex set which can be represented as
\begin{equation}\label{eq:P_def}
P=\{x\in\R^d \colon M x \leq b\},
\end{equation}
where $M$ is an $m\times d$-matrix and $b\in \R^m$.
A bounded polyhedral set is called a \emph{polytope}. A \emph{polyhedral cone} is an intersection of finitely many closed half-spaces whose boundaries contain the origin. If not otherwise stated, polyhedral sets, cones and polytopes are assumed to be nonempty.

A polyhedral set $P$ is called \emph{line-free} if it does not contain a line, i.e.\ a set of the form $\{x+\lambda y\colon \lambda\in\R\}$, $x\in\R^d$, $y\in \R^d\setminus\{0\}$.  In general, the \emph{lineality space} of $P$, represented in the form~\eqref{eq:P_def}, is $U_P=\{x\in\R^d \colon M x = 0\}$. Then, $P$ is line-free if its lineality space is $\{0\}$, or, equivalently,  if the matrix $M$ has rank $d$.  Every polyhedral set has an orthogonal decomposition $P=P_0+U_P$, where $P_0$ is a line-free polyhedral set; see~\cite[Section~7.2]{Pad} or \cite[Lemma 1.4.2]{S14}.

The finite set of faces of a polyhedral set $P$ is denoted by $\FF(P)$, which includes $P$ itself but does not include the empty face. For the notion and basic properties of a face of a polyhedral set we refer to~\cite[Section~7.2.1]{Pad} or \cite{Rock} and, for general convex sets, to  \cite[\S 2.1 and \S 2.4]{S14}. The face structure of a polyhedral set is much simpler than that of a general convex set. In particular, faces are always support sets.  The {\em normal cone} of $P$ at a point $x\in P$ is defined as
\begin{equation}\label{eq:normal_cone_def}
N(P,x)=\{u\in\R^d: \langle u,z-x\rangle \le 0 \text{ for all }z\in P\},
\end{equation}
where $\langle\cdot\,,\cdot\rangle$ denotes the scalar product of the underlying Euclidean space $\R^d$.
For a face $F\in \FF(P)$, the normal cone $N(P,x)$ does not depend on the choice of a point $x\in \relint F$ (the relative interior of $F$), hence it is denoted by  $N(P,F)$ and referred to as the normal cone of $P$ at the face $F$.
Note that $N(P,F)$ is a closed polyhedral cone and $N(P,P) = \{0\}$ if $\text{dim}(P)=d$.
In general, $N(P,P)=\BL(P)^\perp$, where $\BL(P)$ is the linear subspace parallel to the affine hull $\aff(P)$ of $P$.

\subsection{Statement of the result} The indicator function of a set $A\subset \R^d$ is denoted by $\I_A$.
For two sets $A,B \subset \R^d$ let
$$
A+B= \{a+b\colon a\in A, b\in B\}, \quad A-B= \{a-b\colon a\in A, b\in B\}.
$$
The aim of the present note is to prove the following Euler-type inclusion-exclusion relation.
\begin{theorem}\label{theo:main}
Let $P\subset\R^d$ be a polyhedral set. Then, the function $\varphi_P:\R^d\to\mathbb{Z}$ defined by
$$
\varphi_P:=\sum_{F\in\FF(P)}(-1)^{\Dim F}\I_{F-N(P,F)}
$$
is constant and takes values in $\{-1,0,1\}$. If $P = P_0 + U_P$, where $U_P$ is the lineality space of $P$ and $P_0$ is a line-free polyhedral set, then  $\varphi_P\equiv (-1)^{\Dim U_P}$ if $P_0$ is bounded and $\varphi_P\equiv 0$ if
$P_0$ is unbounded.
\end{theorem}

In particular, if $P$ is a polytope, then $\varphi_P\equiv 1$, whereas for an unbounded line-free polyhedral set we have $\varphi_P\equiv 0$.

\subsection{Comments}
Previously, the function $\varphi_P$ was known to be constant everywhere outside some exceptional set
of Lebesgue measure $0$. For polyhedral cones, this statement was given without proof by McMullen in~\cite[p.~249]{McM}.
Proofs for polyhedral cones can be found in the book of Schneider and Weil~\cite{SW}, see the proof of Theorem~6.5.5  there, and in the PhD Thesis of Glasauer~\cite{GlasPhD}, see Hilfssatz 4.3.2 there. Both proofs follow an approach suggested by McMullen. It was conjectured in~\cite{GlasPhD} that the formula should hold without the need to exclude an  exceptional set. A proof of the formula for polytopes was given by Glasauer in~\cite{Glas}.  An extension to Minkowski geometry with a general convex gauge body was obtained by Hug~\cite[Corollary~2.25 on p.~89]{Hug}. In all these results, an exceptional set of measure $0$ is present. Our contribution is to remove such an exceptional set. Shortly before this paper was completed, we became aware of the preprint of Schneider~\cite{Schneider}, where the exceptional set was removed for polyhedral cones. The method used in the present paper is different from Schneider's approach and yields a result valid in the more general setting of polyhedral sets. In fact, reducing the general case of polyhedral sets to polyhedral cones is the core of our proof. After such a reduction has been accomplished, the result for cones follows quickly.

The starting point of the present paper was the observation that the relations stated in Theorem~\ref{theo:main} are  similar to the inclusion-exclusion identities for convex hulls obtained by Cowan~\cite{Cow1,Cow2}. Cowan proved his identities outside  an exceptional set of measure zero.  Recently, the exceptional set was removed in~\cite{KLZ}. The method of~\cite{KLZ}  (based on an extended Euler relation) is quite different from the approach of the present paper.

Let us mention two ``topological'' interpretations of Theorem~\ref{theo:main}. It is well known that
$$
\sum_{F\in \FF(P)} \I_{\relint F + N(P,F)} = 1.
$$
Consider a piecewise linear map $\Psi:\R^d \to\R^d$ defined by $f(x+u) = x-u$, where $x\in \relint F$, $u\in N(P,F)$. This map reflects each set  $\relint F+N(P,F)$ at the corresponding face $F$. The restriction of $\Psi$ to $P$ is the identity map. The map $\Psi$  is continuous but not everywhere smooth. Let $J_\Psi$ be the Jacobian of $\Psi$ (whenever it is defined).  For the ``regular values'' $z$ which do not belong to the boundaries of the sets $F-N(P,F)$, we may define the  degree of $\Psi$ as
\begin{align*}
(\deg \Psi) (z) &=  \sum_{a\in \Psi^{-1} (z)} \sgn \det J_\Psi(a) =
 \sum_{F\in\FF (P)} (-1)^{d-\dim F} \I_{F-N(P,F)}(z) \\
&= (-1)^d \varphi_P(z).
\end{align*}
To give another interpretation, consider the normal bundle of $P$ defined as
$$
\text{NB}(P) = \bigcup_{F\in\FF (P)} \{(x,u)\colon x\in F, u\in N(P,F)\} \subset \R^d \times \R^d.
$$
This is the union set of the normal manifold (normal fan) as defined
in \cite{Robinson92,Scholtes94,Ziegler96,LuRobinson08,Scholtes2012}.
For regular values $z$ one can interpret $\varphi_P(z)$ as the ``intersection index'' between $\text{NB}(P)$ and the $d$-dimensional linear subspace
$$
L_z = \{(x,y)\in \R^d \times \R^d \colon x-y = z\}.
$$
Comparing to classical results in differentiable topology, it is not surprising that $\varphi_P(z)$ stays constant for regular points $z\in\R^d$. The main contribution of the present note is the analysis of the non-regular values.

\subsection{Notation}
The rest of the paper is devoted to the proof of Theorem~\ref{theo:main}. Let us fix some notation. Let $\inter K$ be the interior of a set $K$ and $\partial K$ its boundary.  Let $\relint K$ be the relative interior of a set $K$, that is the interior with respect to its affine hull $\aff K$. If $A\subset\R^d$, then we write $\text{conv}(A)$ for the convex hull  and $\text{pos}(A)$ for the positive hull of $A$.
For a convex set $K\subset \R^d$, let $\BL(K)$ denote the linear subspace parallel to the
affine subspace $\aff K$. Let $B(z,\varepsilon) = \{y\in \R^d\colon \|z-y\| \le \varepsilon\}$ be the closed ball of radius $\varepsilon$ around $z\in\R^d$, where $\|\cdot\|$ denotes the Euclidean norm.

\section{Proof of Theorem~\ref{theo:main}}
%\vspace*{2mm}
\noindent
\textsc{Step 0.}
%We prove (a) and (b) since (c) follows from these ones by considering direct products.
We use induction over the dimension $d$.
In Steps 1--7 below, we prove the following two claims.

\begin{claim}[Reduction to cones]\label{claim:1}
If the function $\varphi_P$ is constant for every polyhedral cone $P$, then it is constant for every polyhedral set $P$.
\end{claim}

\begin{claim}[Induction for cones]\label{claim:2}
If the function $\varphi_P$ is constant for every polyhedral cone $P$ in dimensions $\leq d-1$, then it is constant for every polyhedral cone $P$ in dimension $\le d$.
\end{claim}

Once Claim~\ref{claim:2} has been established, it follows by induction over $d$ that $\varphi_P$ is constant for every polyhedral cone in any dimension. The induction assumption in dimension $d=1$ is easily checked because there are just the following cones: $\{0\}$, $[0,\infty)$, $(-\infty,0]$, $\R$. In fact, for a polyhedral cone $P\subset\R^d$  which is not a linear subspace,  we even have $\varphi_P\equiv 0$ because
\begin{equation}\label{eq:euler_for_cones}
\varphi_P(0) = \sum_{F\in\FF(P)}(-1)^{\Dim F}\I_{F-N(P,F)}(0) = \sum_{F\in\FF(P)}(-1)^{\Dim F}
= 0
\end{equation}
by the Euler relation for polyhedral cones; see, e.g., \cite[Theorem~2.1]{AmelLotz}.  Given that $\varphi_P$ is constant for polyhedral cones, Claim~\ref{claim:1} implies that $\varphi_P$ is constant for all polyhedral sets. It remains to determine the value of this constant, which is done by exhibiting at least one point for which an explicit computation is possible.

\vspace*{2mm}
\noindent
\textit{Case 1.}
Let $P$ be a nonempty bounded polyhedral set. Take any vertex $p$ of $P$ and let $u\in \inter N(P,\{p\})$.
We claim that for each $F\in \FF(P)$, $F\neq\{p\}$, there is some $\lambda_F>0$ such that  $p-\lambda u\notin
F-N(P,F)$ for $\lambda>\lambda_F$. If not, then there is an increasing sequence $(\lambda_i)_{i\in\N}$ in $(0,\infty)$ with
$\lambda_i\to\infty$ for $i\to\infty$ such
that $p-\lambda_i u=f_{i}-u_{i}$ with $f_i\in F$ and $u_{i}\in N(P,F)$, for all $i\in\N$. Hence
$u-u_{i}/\lambda_{i}=(p-f_{i})/\lambda_i\to 0$
as $i\to\infty$ (here we use that $F\subset P$ is bounded). Since $u\in \inter N(P,\{p\})$, it follows that $u_{i}\in
\lambda_i\inter N(P,\{p\})=\inter N(P,\{p\})$ if $i$ is large enough. Since $\inter N(P,\{p\})\cap N(P,F)=\varnothing$ if $F\neq \{p\}$, we arrive
at a contradiction. Hence, choosing $\lambda>0$ larger than $\max\{\lambda_F:F\in \FF(P), F\neq\{p\}\}$, we get
$\varphi_P(p-\lambda u)=(-1)^{0}=1$.

\vspace*{2mm}
\noindent
\textit{Case 2.}
Let $P$ be an unbounded, line-free polyhedral set. Then there are points $a_1,\ldots,a_k\in\R^d$ and
vectors $b_1,\ldots,b_m\in
\R^d\setminus\{0\}$ such that
$$
P=\text{conv}\{a_1,\ldots,a_k\}+\text{pos}\{b_1,\ldots,b_m\},
$$
where $k,m\in\N$ (see \cite[Theorem 19.1]{Rock} or~\cite[7.3(d) on p.~149]{Pad}). Then $C:=\text{pos}\{b_1,\ldots,b_m\}$ is a closed convex cone and $C^*:=C\setminus\{0\}$ is nonempty and
convex, since $P$ is line-free. The reflected polar cone $-C^\circ$ of $C$ is also convex and nonempty. We have $C\cap
(-C^\circ)\neq \{0\}$, since otherwise $C^*\cap   (-C^\circ)=\varnothing$ and a separation argument then yields a vector $u\neq 0$ such that $C^*\subset u^-:=\{z\in\R^d:\langle z,u\rangle\le 0\}$ and $C^\circ\subset u^-$. But then $u\in C^\circ\subset u^-$, and thus $u=0$, a contradiction. Hence there is some $y\in C^\circ \setminus\{0\}$
with $-y\in C$. Alternatively, the fact that $C\cap (-C^\circ)\neq \{0\}$  could be deduced from the generalized Farkas lemma~\cite[Lemma~2.3]{AmelLotz}.

We can assume that $\langle a_1,y\rangle=\max\{\langle a_i,y\rangle:i=1,\ldots,k\}$. From $y\in C^\circ$  we conclude that
$\langle y,b_j\rangle \le 0$ for $j=1,\ldots,m$. Moreover, if $p\in P$ and $u\in N(P,p)$, then
$\langle p-(a_1-\lambda y),u\rangle\ge 0$ for all $\lambda\ge 0$, hence $\langle p-a_1,u\rangle+\lambda\langle y,u\rangle\ge 0$
for all $\lambda\ge 0$ which implies that $\langle y,u\rangle\ge 0$.

Now we can conclude that $a_1+y\notin F-N(P,F)$ for all $F\in \mathcal{F}(P)$. In fact, assume that there are
$F\in \mathcal{F}(P)$, $f\in F$ and $u\in N(P,F)$ such that
$a_1+y=f-u$. Since $f\in P$, there are $\lambda_i,\mu_j\ge 0$ for $i=1,\ldots,k$ and $j=1,\ldots,m$ with $\lambda_1+\cdots+\lambda_k=1$ such that $f=\sum_{i=1}^k\lambda_ia_i+\sum_{j=1}^m\mu_jb_j$. Now we get
$$
\langle f-u,y\rangle=\sum_{i=1}^k\lambda_i\langle a_i,y\rangle+\sum_{j=1}^m\mu_j\langle b_j,y\rangle-\langle u,y\rangle\\
\le \langle a_1,y\rangle< \langle a_1+y,y\rangle,
$$
a contradiction.

It follows that $a_1+y$ is not contained in a set of the form $F-N(P,F)$, $F\in\FF (P)$, and hence, $\varphi_P(a_1+y) = 0$.

\vspace*{2mm}
\noindent
\textit{Case 3.}
Finally, if $P$ is not line-free, then $P=P_0+U_P$, where $P_0$ is a line-free polyhedral set and $U_P\subset\BL(P)$
is the lineality space of $P$. We can choose $P_0$ such that $\BL(P_0)= \BL(P)\cap U_P^\perp$
and hence $U_P=\BL(P_0)^\perp\cap \BL(P)$. First, we observe that
$F\in\mathcal{F}(P)$ if and only if there is some (uniquely determined) $F_0\in \mathcal{F}(P_0)$ with $F=F_0+U_P$.
Conversely, $F_0+U_P\in\mathcal{F}(P)$ for every $F_0\in\mathcal{F}(P_0)$. Let
$N_{\BL(P_0)}(P_0,F_0)\subset\BL(P_0)= U_P^\perp\cap\BL(P)=(U_P+\BL(P)^\perp)^\perp$ be
the normal cone of $P_0$ at its face $F_0$ with respect to $\aff P_0$ as the ambient space.
In this situation, using \cite[Theorem 2.2.1 (a)]{S14} for the second and \cite[(2.5)]{S14}
for the third equation,
we get
\begin{align*}
N(P,F)&=N(P_0+U_P,F_0+U_P)=N(P_0,F_0)\cap U_P^\perp\\
&=(N_{\BL(P_0)}(P_0,F_0)+\BL(P_0)^\perp)\cap U_P^\perp\\
&=(N_{\BL(P_0)}(P_0,F_0)+\BL(P)^\perp+U_P)\cap U_P^\perp\\
&=N_{\BL(P_0)}(P_0,F_0)+\BL(P)^\perp\subset U_P^\perp.
\end{align*}
If  $x\in\R^d$,
then there are uniquely determined $x_0\in\aff P_0$, $u\in U_P$, and $v\in\BL(P)^\perp$ such that $x=x_0+u+v$, and hence
\begin{align*}
\varphi_P(x)&=\sum_{F\in \FF(P)}(-1)^{\dim F} \I_{\relint F - N(P,F)}(x)\\
&=\sum_{F_0\in \FF(P_0)}(-1)^{\dim F_0+\dim U_P} \I_{\relint F_0+U_P - N_{\BL(P_0)}(P_0,F_0)+\BL(P)^\perp}(x_0+u+v)\\
&=(-1)^{\dim U_P}\sum_{F_0\in \FF(P_0)}(-1)^{\dim F_0} \I_{\relint F_0 - N_{\BL(P_0)}(P_0,F_0)}(x_0).
\end{align*}
The assertion now follows from the preceding two cases.

\bigskip

In the following we prove Claims~\ref{claim:1} and Claim~\ref{claim:2}. The proofs of both claims will be parallel.

\vspace*{2mm}
\noindent
\textsc{Step 1.} In the following, let $P\subset \R^d$ be a  polyhedral set. Let $x\in\R^d$ be arbitrarily chosen. Our aim is to show that $\varphi_P$ is constant in a sufficiently small neighbourhood of $x$. We start with some preparations, with the aim of splitting the summation involved in $\varphi_P$ into two parts, one of which is easy to treat.
If $F\in\FF(P)$ is such that $x\notin \partial(F-N(P,F))$,
then
$$
\I_{F-N(P,F)}(x)=\I_{F-N(P,F)}(y) \quad \text{ for all } y\in B(x,\varepsilon),
$$
provided $\varepsilon>0$ is small enough. To see this, observe that the set $\R^d\setminus\partial(F-N(P,F))$ is open.

\vspace*{2mm}
\noindent
\textsc{Step 2.} Now we suppose  $x\in \partial(F-N(P,F))$ for some $F\in\FF(P)$. Our aim is to show that the set
\begin{equation}\label{eq:def_S_x}
S_x:=\{(G,H)\in\FF(P)\times \FF(P)\colon  G\subset H, G\neq H,x\in \relint G-\relint N(P,H)\}
\end{equation}
is nonempty.
 By basic properties of
faces of closed convex sets (see \cite[Sec.~2.1, Theorem~2.1.2]{S14}) we have a disjoint decomposition
$$
F=\bigcupdot_{\substack{G\in \FF(P)\\ G\subset F}}\relint G.
$$
It follows from \cite[Theorem 2.1.2]{S14} and \cite[Proposition 1]{LuRobinson08}, based on \cite{Scholtes94} or its  reprint \cite[Chapters 2.1 and 2.4]{Scholtes2012}, that
$$
N(P,F)=\bigcupdot_{\substack{H\in \FF(P)\\ F\subset H}}\relint N(P,H).
$$
For the sake of convenience we provide another argument based on duality. For this, let $W$ be a polytope with
$\inter W\cap \relint F\neq \varnothing$. By \cite[Theorem 2.2.1 (b)]{S14}, we have
$
N(P,F)=N(P\cap W,F\cap W)$ and $ N(P,H)=N(P\cap W,H\cap W)$
for any $H\in\mathcal{F}(P)$ with $F\subset H$. Clearly, $H\cap W\in\mathcal{F}(P\cap W)$ and any face of $P\cap W$
is of this form.  Then, using also \cite[(2.5)]{S14} twice and writing $L=\BL(P)$, we get
\begin{align*}
N(P,F)&=N(P\cap W,F\cap W)=N_L(P\cap W,F\cap W)+L^\perp\\
&=\bigcupdot_{\substack{H\in \FF(P)\\ F\subset H}}\relint N_L(P\cap W,H\cap W) +L^\perp\\
&=\bigcupdot_{\substack{H\in \FF(P)\\ F\subset H}}\relint (N_L(P\cap W,H\cap W) +L^\perp)\\
&=\bigcupdot_{\substack{H\in \FF(P)\\ F\subset H}}\relint N(P\cap W,H\cap W)\\
&=\bigcupdot_{\substack{H\in \FF(P)\\ F\subset H}}\relint N(P,H),
\end{align*}
where the main step is the third equality which is the reduction to full dimensional polytopes. To prove the
assertion in this special case, we use \cite[Theorem 2.4.9]{S14} to see that
$$
N(P,F)=\pos\{u_i:F\subset F_i,i\in[n]\},
$$
where $[n]=\{1,\ldots,n\}$, $F_1,\ldots,F_n$ are the facets of $P$ and $u_i\in\mathbb{S}^{d-1}$ is
the exterior unit normal vector of $F_i$ for $i=1,\ldots,n$. The proof of \cite[Theorem 2.4.9]{S14} also
shows that $S\in\mathcal{F}(N(P,F))$ if and only if there is a set $J\subset I=\{i\in[n]:F\subset F_i\}$
with $S=\pos\{u_j:j\in J\}$, and then $H:=\cap \{F_j:j\in J\}$ is the uniquely determined face of $P$ with
$F\subset H$ and $S=N(P,H)$. Conversely, if $H\in\mathcal{F}(P)$ with $F\subset H$  then $N(P,H)$ is a face of
$N(P,F)$. Now the assertion follows from \cite[Theorem 2.1.2]{S14}.

Since $\BL(F)$ and $\BL(N(P,F))$ are complementary linear subspaces, it follows that
\begin{equation}\label{eq:F_N_P_F_faces}
F-N(P,F)=\bigcupdot_{\substack{G,H\in \FF(P)\\ G\subset F\subset H}}\left( \relint G-\relint N(P,H)\right).
\end{equation}
is the decomposition of the $d$-dimensional polyhedral set $F-N(P,F)$ into the relative interiors of its faces. Here we use that $\relint (A+B)=\relint A+\relint B$ for all convex sets $A,B\subset \R^d$
(see \cite[Corollary 6.6.2]{Rock}). In particular,
$$\inter(F-N(P,F))=\relint F-\relint N(P,F)$$
and
$$
\partial(F-N(P,F))=\bigcupdot_{\substack{G,H\in \FF(P)\\ G\subset F\subset H, G\neq H}}\left( \relint G-\relint N(P,H)\right).
$$
It follows that the set $S_x$ is nonempty.

\vspace*{2mm}
\noindent
\textsc{Step 3.}
For $G,H\in\FF(P)$ with $G\subset H$ we consider the ``interval''
$$
I(G,H):=\{F\in \FF(P):G\subset F\subset H\}.
$$

\begin{lemma}\label{Lem1}
If $(G_1,H_1)\in S_x$, $(G_2,H_2)\in S_x$ and $(G_1,H_1)\neq (G_2,H_2)$, then
$$
I(G_1,H_1)\cap I(G_2,H_2)=\varnothing.
$$
\end{lemma}

\begin{proof}
We proceed by contradiction. For this we assume that there is  a face $F\in I(G_1,H_1)\cap I(G_2,H_2)$. Then
$G_i\subset F\subset H_i$ and
$$
x\in \relint G_i-\relint N(P,H_i)\subset F-N(P,F),
$$
for $i=1,2$.
Hence $x=g_i-v_i$  with $g_i\in \relint G_i\subset F$ and $v_i\in \relint N(P,H_i)\subset N(P,F)$ for $i=1,2$. From
$g_1-g_2=v_1-v_2\in \BL(F)\cap \BL(N(P,F))=\{0\}$
we see that $g_1=g_2$ and $v_1=v_2$. This implies that $\relint G_1=\relint G_2$ and
$\relint N(P,H_1)=\relint N(P,H_2)$, and thus $G_1=G_2$ and $H_1=H_2$.
\end{proof}

\vspace*{2mm}
\noindent
\textsc{Step 4.}
After these preparations, for $(G,H)\in S_x$ and $y\in\R^d$, we define
$$
\varphi_{G,H}(y):=\sum_{F\in I(G,H)}(-1)^{\dim F}\I_{F-N(P,F)}(y).
$$
Then Lemma \ref{Lem1} implies that
$$
\varphi_P(y)=\sum_{(G,H)\in S_x}\varphi_{G,H}(y)+\sum_{F\in C_x}(-1)^{\dim F}\I_{F-N(P,F)}(y),
$$
where
$$
C_x:=\FF(P)\setminus\bigcupdot_{(G,H)\in S_x}I(G,H).
$$
If $F\in C_x$, then $x\notin \partial (F-N(P,F))$ as we have seen in Step~2, and therefore $\I_{F-N(P,F)}$ is
constant in a sufficiently small neighbourhood of $x$ by Step~1.
Therefore, it is sufficient to show that for all $(G,H)\in S_x$ the
function $\varphi_{G,H}$ is constant in a sufficiently small neighbourhood of $x$.

\vspace*{2mm}
\noindent
\textsc{Step 5.}
So let $(G,H)\in S_x$ be fixed and consider
$\varphi_{G,H}(y)$ for $y\in B(x,\varepsilon)$.
Put $L_1:=\BL(G)$, $L_2:=\BL(N(P,H))$ and $L_3:=(L_1+L_2)^\perp=L_1^\perp\cap \BL(H)$. Thus, we have an orthogonal decomposition $\R^d = L_1 + L_2 + L_3$.
Further, recalling~\eqref{eq:def_S_x}, we have
$x=x_1+x_2$ with uniquely determined $x_1\in\relint G$ and $x_2\in -\relint N(P,H)$.
If $y\in\R^d$ and $y\in B(x,\varepsilon)$, then $y=x+\Delta_1+\Delta_2+\Delta_3$ with $\Delta_i\in L_i$ and $\|\Delta_i\|\le \varepsilon$
for $i=1,2,3$.

\begin{lemma}\label{Lem0}
Let $T\subset\R^d$ be a polyhedral set, $C$ a face of $T$ and $z\in\relint C$. Then there is some $\varepsilon>0$
such that $\I_T(z_1)=\I_T(z_2)$ for all $z_1,z_2\in B(z,\varepsilon)$ with $z_2-z_1\in \BL(C)$.
\end{lemma}
\begin{proof}
The system of inequalities defining $P$ can be written in the form
$$
\langle u_1, y\rangle \leq \alpha_1,\ldots,   \langle u_m, y\rangle \leq \alpha_m
$$
for some $u_1,\ldots,u_m\in\R^d\setminus\{0\}$ and $\alpha_1,\ldots,\alpha_m\in\R$. Faces of $T$ are obtained by turning some of these inequalities into equalities. Without loss of generality, let us assume that $\relint C$ is given by
$$
\langle u_1, y\rangle < \alpha_1,\ldots,   \langle u_l, y\rangle < \alpha_l, \quad
\langle u_{l+1}, y\rangle = \alpha_{l+1}, \ldots,   \langle u_{m}, y\rangle = \alpha_m,
$$
for some $0\leq l\leq m$. Note that ignoring the strict inequalities, we obtain a system defining $\aff C$. Let now $z_1,z_2$ be as in the statement of the lemma. Since linear functions are continuous, we can find some $\varepsilon>0$ such that for all $z_1,z_2\in B(z,\varepsilon)$ we have
$$
\langle u_1, z_i\rangle < \alpha_1,\ldots,   \langle u_l, z_i\rangle < \alpha_l, \quad i=1,2.
$$
Since $z_1-z_2\in \BL(C)$, we also have
$$
\langle u_{l+1}, z_1\rangle = \langle u_{l+1}, z_2\rangle, \ldots,   \langle u_{m}, z_1\rangle = \langle u_{m}, z_2\rangle,
$$
which implies that $z_1\in T$ if and only if $z_2\in T$.
\end{proof}

We apply Lemma \ref{Lem0} with $T=F-N(P,F)$, $C=G-N(P,H)$ (which is a face of $T$ by~\eqref{eq:F_N_P_F_faces}), $z=x$, $z_1=x+\Delta_3$ and $z_2=y$,
where $z_2-z_1=\Delta_1+\Delta_2\in \BL(C)=L_1+L_2$ and $z_1,z_2\in B(x,\varepsilon)$. We conclude that
$\I_{F-N(P,F)}(x+\Delta_3)=\I_{F-N(P,F)}(y)$. This shows that $\varphi_{G,H}(x+\Delta_3)=\varphi_{G,H}(y)$ whenever
$y\in B(x,\varepsilon)$ and $\varepsilon>0$ is small enough.

\vspace*{2mm}
\noindent
\textsc{Step 6.}
In order to complete the proof that
$\varphi_{G,H}(x+w)$  is independent of $w\in L_3$ provided $\|w\|\le \varepsilon$ and $\varepsilon>0$ is sufficiently small, we need some preparation. Let $F$ be any face in $I(G,H)$.
Then we consider the polyhedral cones
\begin{align*}
F^*&:=\pos\left((F-x_1)\cap \BL(G)^\perp\right)=\bigcup_{t>0}\left(t(F-x_1)\cap \BL(G)^\perp\right) \subset
\BL(F)\cap\BL(G)^\perp\subset L_3,\\
H^*&:=\pos\left((H-x_1)\cap \BL(G)^\perp\right)\subset L_3.
\end{align*}
The idea is that we factor $G$ out and that in a small neighborhood of $x_1$, faces look like cones. 
We claim that $F^*$ is a face of the cone $H^*$ and, conversely, all faces of $H^*$ are of the form $F^*$ 
for some $F\in I(G,H)$.
Also,  $\dim F^*=\dim F-\dim G$, $\dim H^*=\dim H-\dim G$. Let $N_{L_3}(H^*,F^*)$ denote the normal
cone of $H^*$ at its face $F^*$ with respect to $L_3$ as the
ambient space.

In order to verify the preceding statements, we put
$$
H^-(u,\alpha):=\{z\in\R^d:\langle z,u\rangle \le \alpha\} \quad\text{and}\quad
  H(u,\alpha):=\{z\in\R^d:\langle z,u\rangle = \alpha\},
$$
that is $ H(u,\alpha)=\partial H^-(u,\alpha)$,
for $u\in\R^d\setminus\{0\}$ and $\alpha\in\R$. The polyhedral {\red set} $P$ is given in the form
$$
P=\bigcap_{i=1}^n H^-(u_i,\alpha_i),
$$
for suitable $n\in\N$, $u_1,\ldots,u_n\in\R^{d}\setminus\{0\}$ and $\alpha_1,\ldots,\alpha_n\in\R$, since
we can assume that $P\neq\R^d$ and hence $n\neq 0$ (the case $P=\R^d$ is trivial). In the following,
intersections over an empty index set are interpreted as $\R^d$. Let
$F$ be a face of $P$.  Then we put $I_F:=\{i\in[n]: F\subset H(u_i,\alpha_i)\}$, $J_F:=\{j\in[n]: G\subset H(u_j,\alpha_j), F\not\subset H(u_j,\alpha_j)\}$, and $R_F:=\{l\in[n]: G\not\subset H(u_l,\alpha_l)\}$. Then
\begin{align*}
F^*&=\bigcup_{t>0}\left[t\left(\bigcap_{i\in I_F}H(u_i,\alpha_i)\cap\bigcap_{j\in J_F}H^-(u_j,\alpha_j)\cap\bigcap_{l\in R_F}H^-(u_l,\alpha_l)-x_1\right)\cap\BL(G)^\perp\right]\\
&=\bigcup_{t>0}\left[\bigcap_{i\in I_F}H(u_i,0)\cap  \bigcap_{j\in J_F}H^-(u_j,0)\cap  \bigcap_{l\in R_F}H^-(u_l,t(\alpha_l-\langle x_1,u_l\rangle) )\cap\BL(G)^\perp\right]\\
&=\bigcap_{i\in I_F}H(u_i,0)\cap  \bigcap_{j\in J_F}H^-(u_j,0)\cap  \bigcup_{t>0}\bigcap_{l\in R_F}H^-(u_l,t(\alpha_l-\langle x_1,u_l\rangle) )\cap\BL(G)^\perp\\
&=\bigcap_{i\in I_F}H(u_i,0)\cap  \bigcap_{j\in J_F}H^-(u_j,0)\cap \BL(G)^\perp\\
&=\bigcap_{j\in J_F}H^-(u_j,0)\cap \BL(F)\cap \BL(G)^\perp,
\end{align*}
where we used that $\langle x_1,u_i\rangle=\alpha_i$ for $i\in I_F\cup J_F$ and $\langle x_1,u_l\rangle<\alpha_l$ for $l\in R_F$.
Similarly, we obtain
\begin{equation}\label{eqnH1}
H^*=\bigcap_{j\in J_H}H^-(u_j,0)\cap \BL(H)\cap \BL(G)^\perp,
\end{equation}
where $J_H:=\{j\in[n]: G\subset H(u_j,\alpha_j), H\not\subset H(u_j,\alpha_j)\}$.
Thus, $J_F\subset J_H$ and $F^*$ is a face of $H^*$ which is obtained from $H^*$ by
turning some of the defining inequalities in \eqref{eqnH1} into equalities.
Conversely, all faces of $H^*$ arise in this way.

\begin{lemma}\label{Lem2}
If $\varepsilon>0$ is sufficiently small and $w\in L_3$ with $\|w\|\le \varepsilon$, then
$x+w\in F-N(P,F)$ if and only if $w\in F^*-N_{L_3}(H^*,F^*)$. Consequently, if $\varepsilon>0$ is
sufficiently small, then $\varphi_{G,H}(x+w) = \varphi_{H^*}(w)$.
\end{lemma}

\begin{proof}
First, assume that $x+w\in F-N(P,F)$. Then there is some $f\in F$ such that
$-(x_1+x_2-f+w)\in N(P,F)\subset \BL(F)^\perp$. Hence $f\in x_1+x_2+w+L(F)^\perp\subset x_1+\BL(G)^\perp$, since
$\BL(G)\subset\BL(F)$, $w\in L_3\subset\BL(G)^\perp$, and $x_2\in \BL(H)^\perp\subset\BL(G)^\perp$. Since also
$f\in x_1+\BL(F)$, we get
%\begin{equation}\label{eqn1}
$$
f\in (x_1+\BL(G)^\perp)\cap (x_1+\BL(F))=x_1+(\BL(G)^\perp\cap\BL(F))\subset x_1 +L_3.
$$
%\end{equation}

We now claim that there is some $f_0\in \relint (F)\cap(x_1+L_3)$.
To see this, we can argue in $\aff F$ and hence assume that $\aff F=\R^d$. Suppose that
$(\inter F) \cap (x_1+L(G)^\perp)=\varnothing$. Then there is a hyperplane $H_0$, bounding the closed
halfspaces $H_0^+,H_0^-$ such that $\inter F\subset H_0^-$ and $x_1+\BL(G)^\perp\subset H_0^+$.
Since $x_1\in G\subset F\subset H_0^-$ and $x_1\in x_1+\BL(G)^\perp\subset H_0^+$, we get $x_1\in H_0$ and thus $x_1+\BL(G)^\perp
\subset H_0$. Moreover, $x_1\in \relint (G)\cap H_0$ and $G\subset F\subset H_0^-$ imply that also $G\subset H_0$.  Thus we arrive
at the contradiction  $d=\dim H_0 = d-1$. 

Choosing $f_0\in\relint(F)\cap (x_1+L_3)$, we get
$$
\langle -(x+w-f),p-f_0\rangle\le 0,\qquad p\in P,
$$
since $-(x-f+w)\in N(P,F)$. Using that $x_2\in \BL(H)^\perp$ and $h-f_0\in\BL(H)$ for $h\in H$ and $f_0\in F\subset H$,
we get
$$
\langle -(x_1+w-f),h-f_0\rangle =\langle -(x+w-f),h-f_0\rangle\le 0
$$
for all $h\in H$.  Thus $x_1+w-f\in -N_{\BL(H)}(H,F)\cap L_3$,  where the lower index indicates the subspace in which the
normal cone is considered. Applying \cite[(2.2)]{S14} in $\BL(H)$ (the polar with respect to $\BL(H)$ of
a cone $C\subset \BL(H)$ is
denoted by $C^\circ_{\BL(H)}$, and similarly for $L_3$) and arguing as in the derivation of \eqref{eqnH1},
we obtain
\begin{align*}
&N_{\BL(H)}(H,F)=\left(\bigcup_{t>0}t(H-f_0)\right)^\circ_{\BL(H)}\\
&\qquad =\left(\bigcap\{H^-(u_i,0)\cap \BL(H):F\subset H(u_i,\alpha_i),i\in[n]\}\right)^\circ_{\BL(H)}\allowdisplaybreaks\\
&\qquad =\left(\BL(G)+\bigcap\{H^-(u_i,0)\cap \BL(H)\cap \BL(G)^\perp:F\subset
H(u_i,\alpha_i),i\in[n]\}\right)^\circ_{\BL(H)}\allowdisplaybreaks\\
&\qquad =\left(\bigcap\{H^-(u_i,0)\cap \BL(H)\cap \BL(G)^\perp:F\subset H(u_i,\alpha_i),i\in[n]\}\right)^\circ_{L_3}\allowdisplaybreaks\\
&\qquad =\left(\bigcap\{H^-(u_i,0)\cap L_3:F\subset H(u_i,\alpha_i),i\in[n]\}\right)^\circ_{L_3}\\
&\qquad =\left(\bigcup_{t>0}t(H^*-(f_0-x_1))\right)^\circ_{L_3}\\
&\qquad =N_{L_3}(H^*,F^*).
\end{align*}
Thus we conclude that $x_1+w-f\in -N_{L_3}(H^*,F^*)$ or
$w\in F^*-N_{L_3}(H^*,F^*)$. (For this direction, we do not have to assume that $\|w\|$ is small.)

Now, we assume that $w\in F^*-N_{L_3}(H^*,F^*)$ and $\|w\|\le \varepsilon$ for a sufficiently small $\varepsilon>0$.
Then $w=f^*-v$, where $f^*\in F^*$ with $\|f^*\|\le \varepsilon$ and $v\in  N_{L_3}(H^*,F^*)$. Further,
$f^*=f_1^*-x_1$ with $f_1^*\in (x_1+\pos(F-x_1))\cap (x_1+L_3)$. If $\|f^*\|\le \varepsilon$ and $\varepsilon>0$
is small enough, then $f_1^*\in F\cap (x_1+L_3)$, hence we have $x_1+w-f_1^*\in -N_{L_3}(H^*,F^*)$ or
$$
\langle -(x_1+w-f_1^*),x_1+h^*-(x_1+f_0^*)\rangle \le 0
$$
for $h^*\in H^*$ and any fixed $f_0^*\in \relint F^*$ which we choose such that $x_1+f_0^*\in \relint F$. Thus
we get
\begin{equation}\label{eqn3}
\langle -(x_1+w-f_1^*),h-f_0\rangle \le 0
\end{equation}
for any $h\in H\cap (x_1+L_3)$ and some (but then also for any) $f_0\in \relint F$. Since
$x_1+w-f^*_1=w-f^*\in \BL(G)^\perp$, \eqref{eqn3} holds for all $h\in H$. Now we put $W=f_0+[-1,1]^d$ and
consider the polytope
$\bar{P}:=P\cap W$. Then $\bar H=H\cap W$ and $\bar F=F\cap W$ are faces of $\bar P$ with $f_0\in \relint \bar F$.
Since $-x_2\in \relint N(P,H)$ and $N(P,H)=N(\bar P,\bar H)$ (by \cite[Theorem 2.2.1 (b)]{S14}), it follows that $\langle -x_2,h-h_0\rangle=0$ for all $h\in\bar H$ and
$ \langle -x_2,p-h_0\rangle<0$ for all $h\in\bar P\setminus \bar H$ (see \cite[(2.26)]{S14}). In particular, writing $\ext\bar P$ for the finite  set of vertices (extreme points) of $\bar P$, we have
$$
\max\{\langle -x_2,e-f_0\rangle:e\in \ext\bar P\setminus \bar H\}=:\varepsilon_1<0,
$$
where we used that $\langle -x_2,e-f_0\rangle=\langle -x_2,e-h_0\rangle$, since $\langle -x_2,f_0-h_0\rangle=0$  and
$\bar F\subset\bar H$. For $e\in\bar P$, we obtain
$$
|\langle -(x_1+w-f_1^*),e-f_0\rangle |\le (\|w\|+\|f^*\|)\|e-f_0\|\le  2\varepsilon\sqrt{d}<-\varepsilon_1,
$$
if $\varepsilon$ is chosen sufficiently small.

If $e\in \ext\bar P\cap \bar H$, then
$$
\langle -(x_1+w-f_1^*)-x_2,e-f_0\rangle=\langle -(x_1+w-f_1^*),e-f_0\rangle+\langle -x_2,e-f_0\rangle \le 0+0=0.
$$
If
$e\in \ext\bar P \setminus \bar H$, then
$$
\langle -(x_1+w-f_1^*)-x_2,e-f_0\rangle=\langle -(x_1+w-f_1^*),e-f_0\rangle+\langle -x_2,e-f_0\rangle \le
-\varepsilon_1+\varepsilon_1=0.
$$
Thus we have
$$
\langle -(x_1+w-f_1^*-x_2),e-f_0\rangle\le 0,
$$
first for all $e\in\bar P$, but then for all $e\in \bar P$, since for any $p\in\bar P$
there are $\lambda_e\ge 0$ with $\sum_{e\in\ext\bar P}\lambda_e=1$ and $p=\sum_{e\in\ext\bar P}\lambda_e e$. Hence
we have $p-f_0= \sum_{e\in\ext\bar P}\lambda_e (e-f_0)$, from which the assertion  follows.

This shows that $-(x+w-f_1^*)\in N(\bar P,\bar F)=N(P,F)$, and therefore
 $x+w\in F-N(P,F)$.
\end{proof}

\vspace*{2mm}
\noindent
\textsc{Step 7.}
Lemma \ref{Lem2} reduces the problem to the case of polyhedral cones, i.e.\ it proves Claim~\ref{claim:1}.
 To prove Claim~\ref{claim:2}, let $P$ be a  polyhedral cone but not a linear subspace. If $x=0$, then $\varphi_P(x)=0$  by the classical Euler relation~\eqref{eq:euler_for_cones}.  If $x\neq 0$, then going through the preceding argument again, we see that we have to show that
$$
\varphi_{H^*}(w) = \sum_{F^*\in\FF(H^*)}(-1)^{\dim F^*}\I_{F^*-N_{H^*}(H^*,F^*)}(w)
$$
is independent of $w\in L_3$. But since $x\neq 0$, $G\subset H$, $G\neq H$ and $x\in \relint G - \relint N(P,H)$,
we must have $\dim G>0$ or $\dim H<d$, since otherwise $0\neq x\in \relint(\{0\})-\relint N(P,P) =\{0\}$, a contradiction. But then $\dim H^*=\dim H - \dim G < d$, so that the induction hypothesis of Claim~\ref{claim:2} can be applied.

\end{document}